\documentclass[12pt]{amsart}

\usepackage{amssymb,mathrsfs,amsmath,amsthm,color,bm,mathtools,bbm,wasysym,subfig,float}
\usepackage[noadjust]{cite}

\usepackage{enumerate}
\usepackage[centering]{geometry}
\allowdisplaybreaks

\theoremstyle{plain}
\newtheorem{thm}{Theorem}[section]

\newtheorem{prop}[thm]{Proposition}
\newtheorem{cor}[thm]{Corollary}
\newtheorem{lem}[thm]{Lemma}

\theoremstyle{remark}
\newtheorem{rem}{Remark}
\numberwithin{equation}{section}

\DeclareMathOperator{\hdim}{\dim_H}

\newcommand{\rc}{\mathcal R}

\newcommand{\N}{\mathbb N}
\newcommand{\R}{\mathbb R}
\newcommand{\Z}{\mathbb Z}
\newcommand{\lm}{\mathcal L}
\renewcommand{\hm}{\mathcal H}
\newcommand{\hc}{\mathcal H_\infty}

\newcommand{\ca}{\mathcal A}

\newcommand{\bd}{\bm \delta}
\newcommand{\bx}{\mathbf{x}}
\newcommand{\ba}{\mathbf{a}}
\newcommand{\by}{\mathbf{y}}

\newcommand{\bb}{\mathbf{b}}

\newcommand{\bu}{\mathbf{u}}
\newcommand{\bq}{\mathbf{q}}

\newcommand{\bv}{\mathbf{v}}

\newcommand{\qaq}{\mathrm{\quad and\quad}}
\newcommand{\be}{\mathbf{e}}

\newcommand{\balpha}{{\bm\alpha}}
\newcommand{\bbeta}{{\bm\beta}}

\begin{document}
		\title[Inhomogeneous Dirichlet non-improvable]{Hausdorff measure of sets of inhomogeneous Dirichlet non-improvable affine forms with weights}
	\author{Yubin He}

	\address{Department of Mathematics, Shantou University, Shantou, Guangdong, 515063, China}

	\email{ybhe@stu.edu.cn}

%
%


	\begin{abstract}
		 Under a reasonable decay assumption on the approximating function, we establish a zero-full law for the Hausdorff measure of sets of inhomogeneous Dirichlet non-improvable affine forms with weights, thereby answering a question posed by Kim and Kim (\S 5.3, Adv. Math., 2022).
	\end{abstract}
	\maketitle
\section{Introduction}\label{s:intro}

Let $m\ge 1$ and $n\ge 1$ be integers and let $M_{m,n}(\R)$ be the space of $m$ by $n$ matrices. The starting point
for the present paper is the following higher dimensional generalization of the classical Dirichlet's theorem, established in 1842.
\begin{thm}[Dirichlet]
	For any $A\in M_{m,n}(\R)$ and $t>1$, there exists $\bq\in\Z^n\setminus\{0\}$ such that
	\begin{equation}\label{eq:Dirichlet}
		\|A\bq\|_\Z^m\le t^{-1}\qaq |\bq|^n<t.
	\end{equation}
\end{thm}
Here, $\|\cdot\|_\Z$ and $|\cdot|$ denote the distance to the nearest integral vector and the supremum norm on $\R^d$ given by $|\bx|=\max_{1\le i\le d}|x_i|$, respectively. Over the past few decades, the above theorem has been applied through its corollary.
\begin{cor}
For any $A\in M_{m,n}(\R)$ there exist infinitely many $\bq\in\Z^n$ such that
\begin{equation}\label{eq:Dirichletcor}
	\|A\bq\|_\Z<|\bq|^{-n/m}.
\end{equation}
\end{cor}
For obvious reasons, the statements in Dirichlet's theorem and its corollary are commonly referred to as {\em uniform} and {\em asymptotic approximation}, respectively. It is natural to ask whether the right-hand sides of \eqref{eq:Dirichlet} and \eqref{eq:Dirichletcor} can be improved by faster decaying functions of $t$ and $\bq$, respectively. Historically, improvements to Dirichlet's corollary predate those to his original theorem. Let $\psi:\R_{\ge 0}\to\R_{\ge 0}$ be a non-negative function. A matrix $A\in M_{m,n}(\R)$ is said to be {\em $\psi$-approximable} if there exist infinitely many $\bq\in\Z^n$ such that
\[	\|A\bq\|_\Z<\psi(|\bq|).\]
Denote by $W_{m,n}(\psi)$ the set of $\psi$-approximable matrices in the unit cube $[0,1]^{mn}$. For monotonic $\psi$, the fundamental Khintchine--Groshev Theorem  provides a criterion for determining whether the Lebesgue measure of $W_{m,n}(\psi)$ is zero or full, depending on whether a certain series converges or diverges.
\begin{thm}[Khintchine--Groshev Theorem]\label{t:BV}
	Let $\psi:\R_{\ge 0}\to\R_{\ge 0}$. Then,
	\[\lm^{mn}\big(W_{m,n}(\psi)\big)=\begin{dcases}
		0&\text{ if $\sum_{q=1}^\infty q^{n-1}\psi(q)^m<\infty$},\\
		1&\text{ if $\sum_{q=1}^\infty q^{n-1}\psi(q)^m=\infty$ and $\psi$ is non-increasing},
	\end{dcases}\]
	where $\lm^{mn}$ denotes the $mn$-dimensional Lebesgue measure.
\end{thm}
\begin{rem}
	It was shown by Beresnevich and Velani \cite[Theorem 5]{BV10} that the monotonic assumption on $\psi$ is actually unnecessary if $mn>1$.
\end{rem}
The Lebesgue measure of $W_{m,n}(\psi)$ will be zero if the series in the theorem above converges. In this case, Hausdorff dimension and Hausdorff $f$-measure (denoted by  $\hdim$ and $\hm^f$, respectively) often provide a meaningful way to distinguish the relative `sizes' of these sets. Another fundamental result, due to Jarn\'ik~\cite{Jarnik31}, provides a criterion analogous to Theorem~\ref{t:BV} for determining the Hausdorff measure of $W_{m,1}(\psi)$. This result was subsequently generalized by Dickinson and Velani \cite{DV97} to arbitrary $n \ge 1$.

Throughout, a continuous, non-decreasing function defined on $\R_{\ge 0}$ and satisfying $f(0)=0$ will be called a {\em dimension function}. For a dimension function $f$ and $s\ge 0$, by $f\preceq s$
we mean that
\begin{equation}\label{eq:prec}
	\frac{f(y)}{y^s}\le \frac{f(x)}{x^s}\quad\text{for any $0<x<y$},
\end{equation}
which implies that
\[\lim_{r\to 0^+}\frac{f(r)}{r^s}>0.\]
The limit may be finite or infinite. If the limit is infinite, then we write $f\prec s$ instead of $f\preceq s$. The notation $s\preceq f$ ($s\prec f$) is defined analogously.
\begin{thm}[{\cite{Jarnik31,DV97}}]\label{t:conAB}
	Let $\psi:\R_{\ge 0}\to\R_{\ge 0}$ be a non-increasing function and let $f$ be a dimension function such that $m(n-1)\prec f\preceq mn$. Then,
	\[\hm^f\big(W_{m,n}(\psi)\big)=\begin{dcases}
		0&\text{ if $\sum_{q=1}^\infty f\bigg(\frac{\psi(q)}{q}\bigg)\bigg(\frac{\psi(q)}{q}\bigg)^{m(1-n)}q^{m+n-1}<\infty$},\\
		\hm^f([0,1]^{mn})&\text{ if $\sum_{q=1}^\infty f\bigg(\frac{\psi(q)}{q}\bigg)\bigg(\frac{\psi(q)}{q}\bigg)^{m(1-n)}q^{m+n-1}=\infty$}.
	\end{dcases}\]
\end{thm}
Let us turn to the improvements to Dirichlet's theorem, which have attracted much attention recently. Let $\psi:\R_{>1}\to\R_{\ge 0}$ be a decreasing function. A matrix $A\in M_{m,n}(\R)$ is said to be {\em $\psi_{m,n}$-Dirichlet improvable} if the system of inequalities
\[\|A\bq\|_{\Z}^m<\psi(t)\qaq|\bq|^n< t\]
has solutions in $\bq\in\Z^n\setminus\{0\}$ for all sufficiently large $t$. Denote by $D_{m,n}(\psi)$ the set of $\psi_{m,n}$-Dirichlet improvable matrices in the unit cube $[0,1]^{mn}$. When $m=n=1$, using the theory of continued fractions, Kleinbock and Wadleigh \cite{KW18} characterized $\psi_{1,1}$-Dirichlet improvable points in terms of the growth of the product of two consecutive partial quotients. This, in turn, enabled them to establish a zero-one law for the Lebesgue measure of $D_{1,1}(\psi)$. Following some of these ideas, Hussain, Kleinbock, Wadleigh and Wang \cite{HKWW18} and Bos, Hussain and Simmons \cite{BHS23} provided Hausdorff measure-theoretic results for the set of $\psi_{1,1}$-Dirichlet non-improvable numbers $D_{1,1}(\psi)^c$. In higher dimensions the machinery of continued fractions is no longer available. Nevertheless, the problem can be reformulated as a {\em shrinking target problem} (see \cite[\S 4]{KW18}), which subsequently allowed Kleinbock Str\"ombergsson and Yu \cite{KSS22} to provide sufficient conditions on $\psi$ for determining whether $D_{m,n}(\psi)$ has zero or full Lebesgue measure. It is worth emphasizing that their conditions differ by a logarithmic factor.

Although the measure-theoretic results for the set of Dirichlet non-improvable matrices remain incomplete, those in the inhomogeneous setting are comparatively well understood. Let $\balpha=(\alpha_1,\dots,\alpha_m)\in(\R_+)^m$ and $\bbeta=(\beta_1,\dots,\beta_n)\in(\R_+)^n$ be satisfying
\[\sum_{i=1}^{m}\alpha_i=1=\sum_{j=1}^{n}\beta_j.\]
Define the $\balpha$-quasinorm of $\bx\in\R^m$ and $\bbeta$-quasinorm of $\by\in\R^n$ by
\[|\bx|_{\balpha}=\max_{1\le i\le m}|x_i|^{\frac{1}{\alpha_i}}\qaq |\by|_\bbeta=\max_{1\le j\le n}|y_j|^{\frac{1}{\beta_j}}.\]
By changing the coordinates if necessary, assume that
\[\beta_1\ge\beta_2\ge\cdots\ge \beta_n.\]
Besides, define
\[\|\bx\|_{\Z,\balpha}=\max_{1\le i\le m}\|x_i\|_{\Z}^{\frac{1}{\alpha_i}}\qaq \|\by\|_{\Z,\bbeta}=\max_{1\le j\le n}\|y_j\|_{\Z}^{\frac{1}{\beta_j}}.\]
Let $\bb\in\R^m$ be given. A matrix $A\in M_{m,n}(\R)$ is said to be $(\psi_{\balpha,\bbeta},\bb)$-Dirichlet if the system of inequalities
\[\|A\bq+\bb\|_{\Z,\balpha}<\psi(t)\qaq|\bq|_\bbeta< t\]
has solutions in $\bq\in\Z^n\setminus\{0\}$ for all sufficiently large $t$.
Denote by $D_{\balpha,\bbeta}^\bb(\psi)$ the set of $(\psi_{\balpha,\bbeta},\bb)$-Dirichlet matrices in the unit cube $[0,1]^{mn}$. If $\balpha=(1/m,\dots,1/m)$ and $\bbeta=(1/n,\dots,1/n)$, then we write $D_{m,n}^\bb(\psi)$ instead of $D_{\balpha,\bbeta}^\bb(\psi)$.
\begin{thm}[{\cite[Theorem 1.4]{Kim22}}]\label{t:Kim}
	Let $\psi$ be a decreasing function with $\lim_{t\to\infty}\psi(t)=0$, and let $0\le s\le mn$. Let $\bb\in\R^m\setminus\Z^m$ be fixed. Then,
	\[\hm^s\big(D_{m,n}^\bb(\psi)^c\big)=\begin{dcases}
		0&\text{ if $\sum_{q=1}^\infty \frac{1}{\psi(q)q^2}\bigg(\frac{q^{\frac1n}}{\psi(q)^{\frac1m}}\bigg)^{mn-s}<\infty$},\\
		\hm^s([0,1]^{mn})&\text{ if $\sum_{q=1}^\infty \frac{1}{\psi(q)q^2}\bigg(\frac{q^{\frac1n}}{\psi(q)^{\frac1m}}\bigg)^{mn-s}=\infty$}.
	\end{dcases}\]
\end{thm}
Given the many recent developments in weighted asymptotic approximation \cite{He25,LLVWZ25,WW21}, Kim and Kim \cite[\S5.3]{Kim22} posed the question of establishing a criterion for the Hausdorff measure of the general set $D_{\balpha,\bbeta}^\bb(\psi)^c$, which is the main purpose of the present paper.
\begin{thm}\label{t:main}
		Let $\psi$ be a decreasing and continuous function with $\lim_{t\to\infty}\psi(t)=0$. Let $\balpha\in(\R_+)^m$ and $\bbeta\in(\R_+)^n$ be two weighted vectors with $\beta_1\ge \cdots\ge\beta_n$.
	 Suppose that there exists a constant $\lambda>1$ such that for any $t\ge 1$,
	\begin{equation}\label{eq:lambda}
		\frac{\psi(t)}{\psi(2t)}\ge\lambda>1.
	\end{equation}
		Let $f\prec mn$ be a dimension function such that $(mn-a)\preceq f\preceq(mn-a+1)$ for some $1\le a\le n-1$. Then, for any $\bb\in\R^m\setminus\Z^m$,
		\[\hm^f\big(D_{\balpha,\bbeta}^\bb(\psi)^c\big)=\begin{dcases}
			0&\text{ if $\sum_{\bu\in\Z^m\setminus\{0\}}\gamma_\bu(\bbeta,f)|\bu|^n<\infty$},\\
			\hm^f([0,1]^{mn})&\text{ if $\sum_{\bu\in\Z^m\setminus\{0\}}\gamma_\bu(\bbeta,f)|\bu|^n=\infty$},
		\end{dcases}\]
		where
		\[\gamma_\bu(\bbeta,f):=\min_{1\le j\le n}\bigg\{ f\bigg(\frac{t(\bu)^{-\beta_j}}{|\bu|}\bigg)\bigg(\frac{t(\bu)^{-\beta_j}}{|\bu|}\bigg)^{(1-m)n}\prod_{\ell=j}^nt(\bu)^{\beta_j-\beta_\ell}\bigg\},\]
		and in turn, for any $\bu\in\Z^m$,
		\[t(\bu):=\inf\{t>0:|u_i|<\psi(t)^{-\alpha_i}\text{ for $1\le i\le m$}\}.\]
\end{thm}
\begin{rem}
	For any $\bu\in\Z^m$, it is straightforward to verify that $t(\bu)=\psi^{-1}(|u_k|^{-\frac{1}{\alpha_k}})$, where $|u_k|^{-\frac{1}{\alpha_k}}$ is the maximum among $|u_1|^{\frac{1}{\alpha_1}},\dots,|u_m|^{\frac{1}{\alpha_m}}$. Naturally, one would like to use this to simplify the summation in the above theorem to a more easily verifiable form. In the special case where $\beta_1=\cdots=\beta_n=1/n$,  $\alpha_1=\cdots=\alpha_m=1/m$ and $f(r)=r^s$, we have for any $\bu\in\Z^m\setminus\{0\}$,
	\[t(\bu)=\psi^{-1}(|\bu|^{-m})\qaq\gamma_\bu(\bbeta,f)=\bigg(\frac{\psi^{-1}(|\bu|^{-m})^{-\frac1n}}{|\bu|}\bigg)^{s-(m-1)n}.\]
	Therefore, the summation becomes
	\[\sum_{u=1}^{\infty}u^{m+n-1}\bigg(\frac{\psi^{-1}(u^{-m})^{-\frac1n}}{u}\bigg)^{s-(m-1)n},\]
	whose convergence or divergence is equivalent to that of the series in Theorem \ref{t:Kim} (see \cite[Lemma 4.2]{Kim22} for a proof). However, in other cases, this turns out to be a nontrivial task.
\end{rem}
\begin{rem}
	The crucial difference between the inhomogeneous case ($\bb\notin\Z^m$) and the homogeneous case ($\bb\in\Z^m$) is that, following the idea of \cite[Lemma 4.1]{Kim22}, the Diophantine transference principle allows the set $D_{\balpha,\bbeta}^\bb(\psi)^c$ in the inhomogeneous setting to be almost reinterpreted as a limsup set defined by neighborhoods of hyperplanes. This perspective naturally transforms the problem into one concerning asymptotic approximation, which enabled Kim and Kim \cite{Kim22} to establish their zero-one law in the unweighted case. We will follow this idea but our approach is primarily motivated by the author's work on weighted asymptotic approximation \cite{He25} and differs from that of \cite{Kim22}. The main ingredient is to construct a $\limsup$ set of full Lebesgue measure that contains a suitable subset of $D_{\balpha,\bbeta}^\bb(\psi)^c$, and then apply the mass transference principle from balls to rectangles to conclude that this subset---and therefore $D_{\balpha,\bbeta}^\bb(\psi)^c$---has full Hausdorff $f$-measure. However, the technical details differ due to the distinct setting.
\end{rem}

This paper is organized as follows. In Section \ref{s:Hausdorff}, we recall several notions and properties of Hausdorff measure and content. Additionally, we present the key tool---the mass transference principle from balls to rectangles---as well as the Hausdorff $f$-content of a hyperrectangle in the same section.
 In Section \ref{s:relation}, we use the Diophantine transference principle to reinterpret $D_{\balpha,\bbeta}^\bb(\psi)^c$  as limsup sets defined by neighborhoods of hyperplanes. In Sections \ref{s:conver} and \ref{s:diver}, we establish the convergence and divergence parts of Theorem \ref{t:main}, respectively.

\section{Hausdorff measure and content}\label{s:Hausdorff}

Throughout, the symbols
$\ll$ and $\gg$ will be used to indicate an inequality with an unspecified positive multiplicative constant. If $a\ll b$ and $a\gg b$, we
write $a\asymp b$ and say that the quantities $a$ and $b$ are comparable.

Let $f$ be a dimension function. For a set $E\subset \R^d$ and $\eta>0$, define
\[\mathcal H_\eta^f(E):=\inf\bigg\{\sum_{i}f(|B_i|):E\subset \bigcup_{i}B_i, \text{ where $B_i$ are balls with $|B_i|\le \eta$}\bigg\},\]
where $|A|$ denotes the diameter of the set $A\subset \R^d$.
The {\em Hausdorff $f$-measure} of $E$ is defined as
\[\hm^f(E):=\lim_{\eta\to 0}\mathcal H_\eta^f(E).\]
When $\eta=\infty$, $\hc^f(E)$ is referred to as the {\em Hausdorff $f$-content} of $E$.

The key tool for proving the divergence part of Theorem \ref{t:main} is the following mass transference principle from balls to open sets. Originally established by Koivusalo and Rams \cite{KR21} for Hausdorff dimension, this mass transference principle was later refined by Zhong \cite{Zh21} to extend to the Hausdorff measure statement.

\begin{thm}[{\cite[Theorem 3.2]{KR21} and \cite[Theorem 1.6]{Zh21}}]\label{t:weaken}
	Let $f$ be a dimension function such that $f\preceq d$. Assume that $\{B_k\}$ is a sequence of balls in $[0,1]^d$ with radii tending to 0, and that $\lm^d(\limsup B_k)=1$. Let $\{E_k\}$ be a sequence of open sets such that $E_k\subset B_k$. If there exists a constant $c>0$ such that for any $k\ge 1$,
	\begin{equation}\label{eq:cont}
		\hc^f (E_k)>c\lm^d(B_k),
	\end{equation}
	then,
	\[\hm^f\Big(\limsup_{k\to\infty} E_k\Big)=\hm^f([0,1]^d).\]
\end{thm}
\begin{rem}
	In fact, a $\limsup$ set satisfying \eqref{eq:cont} has the so-called large intersection property, as established in \cite[Theorem 2.4]{He25adv}, meaning that any countable intersection of sets with this property still has full Hausdorff  $f$-measure.
\end{rem}
With this result now at our disposal, the strategy for proving the divergence part of Theorem \ref{t:main} is to verify certain hyperrectanlges under consideration satisfying an appropriate Hausdorff $f$-content bound. To this end, the following estimate for the Hausdorff $f$-content of a hyperrectangle will play a crucial role.
\begin{prop}[{\cite[(2.3)]{He25}}]\label{p:rec}
	Let $R\subset\R^d$ be a hyperrectangle with side lengths
	\[a_1\ge a_2\ge\cdots\ge a_d>0.\]
	Suppose that $f$ is a dimension function such that $k\preceq f\preceq (k+1)$ for some $0\le k\le d-1 $. Then,
		\[\hc^f(R)\asymp\min_{1\le i\le d}\bigg\{f(a_i)\prod_{j=1}^{i-1}\dfrac{a_j}{a_i}\bigg\},\]
	where the unspecified constant depends on $d$ only.
\end{prop}
\section{Relation to weighted asymptotic approximation}\label{s:relation}
In this section, we apply the Diophantine transference principle to reinterpret $D_{\balpha,\bbeta}^\bb(\psi)^c$  as limsup sets defined by neighborhoods of hyperplanes. This idea originates from \cite[Lemma 4.1]{Kim22}.

For any matrix $A\in M_{m,n}(\R)$, we use the notation $A_{*,j}$ to denote the $j$th column of $A$. For two vectors $\ba, \bb$ with the same dimension, denote by $\ba \cdot \bb$ the standard inner product.
\begin{thm}[Diophantine transference principle {\cite[Theorem XVII]{Cassels57}}]\label{t:transference}
	Let $\bb\in\R^m$ and $A=(A_{ij})_{m\times n}\in M_{m,n}(\R)$ be given. Let $c>0$ and $t>1$.
	\begin{enumerate}
		\item If there exists $\bq\in\Z^n$ such that
		\[\|A\bq+\bb\|_{\Z,\balpha}\le c\qaq |\bq|_{\bbeta}\le t,\]
		then for any integral $\bu\in\Z^m$,
		\[\|\bu\cdot\bb\|_{\Z}\le (m+n)\max\Big\{\max_{1\le j\le n}t^{\beta_j}\|A_{*,j}\cdot\bu\|_\Z,\max_{1\le i\le m}c^{\alpha_i}|u_i|\Big\}.\]
		\item If for any integral $\bu\in\Z^m$,
		\[\|\bu\cdot\bb\|_{\Z}\le 2^{m+n-1}\big((m+n)!\big)^{-2}\max\Big\{\max_{1\le j\le n}t^{\beta_j}\|A_{*,j}\cdot\bu\|_\Z,\max_{1\le i\le m}c^{\alpha_i}|u_i|\Big\},\]
		then there exists $\bq\in\Z^n$ such that
		\[\|A\bq+\bb\|_{\Z,\balpha}\le c\qaq |\bq|_{\bbeta}\le t.\]
	\end{enumerate}
\end{thm}
The theorem offers a meaningful connection between inhomogeneous approximation and the Diophantine properties of $A^\top$, but yields nothing non-trivial when $\bb \in \mathbb{Z}^m$. As a corollary, we relate inhomogeneous Dirichlet non-improvability to asymptotic approximability. Recall from Theorem \ref{t:main} that for any $\bu\in\Z^m$,
\[t(\bu):=\inf\{t>0:|u_i|<\psi(t)^{-\alpha_i}\text{ for $1\le i\le m$}\}.\]
\begin{cor}\label{c:multivariable}
	Let $\bb\in\R^m\setminus\Z^m$ and $A=(A_{ij})_{m\times n}\in M_{m,n}(\R)$ be given. Suppose that there exists a constant $\lambda>1$ such that for any $t\ge 1$,
	\begin{equation}\label{eq:lambda'}
		\frac{\psi(t)}{\psi(2t)}\ge\lambda>1.
	\end{equation}
	\begin{enumerate}
		\item If there exist infinitely many $\bu\in\Z^m$ such that
		\begin{equation}\label{eq:condition1}
			\|A_{*,j}\cdot\bu\|_\Z<\frac{\|\bu\cdot\bb\|_\Z}{m+n}\cdot \big(2^{\tau(\bb,\bu)} t(\bu)\big)^{-\beta_j}\quad\text{for $1\le j\le n$},
		\end{equation}
		then $A\in D_{\balpha,\bbeta}^\bb(\psi)^c$, where $\tau(\bb,\bu)$ is the smallest integer for which
		\[\lambda^{-\tau(\bb,\bu)\alpha_i}<\frac{\|\bu\cdot\bb\|_\Z}{m+n}\quad\text{for $1\le i\le m$.}\]
		\item If $A\in D_{\balpha,\bbeta}^\bb(\psi)^c$, then there exist infinitely many $\bu\in\Z^m$ such that
		\[\|A_{*,j}\cdot\bu\|_\Z<\big((m+n)!\big)^2\cdot \big(2^{-\tau}t(\bu)\big)^{-\beta_j}\quad\text{for $1\le j\le n$},\]
		where $\tau$ is the smallest integer for which
		\[\lambda^{\tau\alpha_i}\ge\big((m+n)!\big)^2\quad\text{for $1\le i\le m$.}\]
			\end{enumerate}
\end{cor}
\begin{proof}
	(1) In view of Theorem \ref{t:transference} (1), it suffices to prove that there exists an unbounded set $t\ge 1$ such that the system
	\begin{equation}\label{eq:u*b>max}
		\|\bu\cdot\bb\|_{\Z}>(m+n)\max\Big\{\max_{1\le j\le n}t^{\beta_j}\|A_{*,j}\cdot\bu\|_\Z,\max_{1\le i\le m}\psi(t)^{\alpha_i}|u_i|\Big\}
	\end{equation}
	has a solution $\bu\in\Z^m$ depending on $t$.

	For any $\bu\in\Z^m$, by the definition of $t(\bu)$ and note that $\psi$ is decreasing and continuous, we have
	\[|u_i|\le\psi\big(t(\bu)\big)^{-\alpha_i}\text{ for $1\le i\le m$, or equivalently }\max_{1\le i\le m}\psi\big(t(\bu)\big)^{\alpha_i}|u_i|\le1.\]
	By the decay assumption on $\psi$ (see \eqref{eq:lambda'}),
	\begin{equation}\label{eq:max2taubu}
		\max_{1\le i\le m}\psi\big(2^{\tau(\bb,\bu)} t(\bu)\big)^{\alpha_i}|u_i|\le \max_{1\le i\le m}\lambda^{-\tau(\bb,\bu)\alpha_i}\psi\big(t(\bu)\big)^{\alpha_i}|u_i|<\frac{\|\bu\cdot\bb\|_\Z}{n+m},
	\end{equation}
	where we use the definition of $\tau(\bb,\bu)$ in the last inequality.
	Let $\bu\in\Z^m$ be such that \eqref{eq:condition1} holds. It follows from \eqref{eq:max2taubu} that \eqref{eq:u*b>max} holds with $t$ replaced by $2^{\tau(\bb,\bu)} t(\bu)$. Since $t(\bu)\to\infty$ as $|\bu|\to\infty$, the assumption in item (1) implies that the set $\{2^{\tau(\bb,\bu)} t(\bu):\bu\text{ satisfies \eqref{eq:condition1}}\}$ is unbounded. Therefore, $A\in D_{\balpha,\bbeta}^\bb(\psi)^c$.

	(2) Since $A\in D_{\balpha,\bbeta}^\bb(\psi)^c$, by Theorem \ref{t:transference} (2), there exists an unbounded set $\{t_k\}_{k\ge 1}$ such that for each $t_k$, there exists $\bu_k = (u_{k,1},\dots,u_{k,m}) \in \Z^m$  such that
	\[\begin{split}
		\max\Big\{\max_{1\le j\le n}t_k^{\beta_j}\|A_{*,j}\cdot\bu_k\|_\Z,\max_{1\le i\le m}\psi(t_k)^{\alpha_i}|u_{k,i}|\Big\}&< \frac{\|\bu\cdot\bb\|_{\Z}\big((m+n)!\big)^2}{2^{m+n-1}}\\
		&<\big((m+n)!\big)^2.
	\end{split}\]
	For any $k\ge 1$,  by the definition of $t(\bu)$ and note that $\psi$ is decreasing and continuous, we have for some $1\le i\le m$,
	\[|u_{k,i}|=\psi\big(t(\bu_k)\big)^{-\alpha_i}\quad\Longleftrightarrow\quad \psi\big(t(\bu_k)\big)^{\alpha_i}|u_{k,i}|=1,\]
	which, by the decay assumption on $\psi$ again, implies that
	\[\psi\big(2^{-\tau}t(\bu_k)\big)^{\alpha_i}|u_{k,i}|\ge \lambda^{\tau\alpha_i}\psi\big(t(\bu_k)\big)^{\alpha_i}|u_{k,i}|\ge \big((m+n)!\big)^2.\]
	Therefore, $2^{-\tau}t(\bu_k)<t_k$ and so
	\[\max_{1\le j\le n}\big(2^{-\tau}t(\bu_k)\big)^{\beta_j}\|A_{*,j}\cdot\bu_k\|_\Z<\max_{1\le j\le n}t_k^{\beta_j}\|A_{*,j}\cdot\bu_k\|_\Z<\big((m+n)!\big)^2,\]
	and the proof is completed.
\end{proof}

\section{Convergence part of Theorem \ref{t:main}}\label{s:conver}
In this section, assume that $\beta_1\ge \cdots\ge\beta_n$ and that
\begin{equation}\label{eq:concon}
	\sum_{\bu\in\Z^m\setminus\{0\}}\gamma_\bu(\bbeta,f)|\bu|^n<\infty,
\end{equation}
where
\[\gamma_\bu(\bbeta,f)=\min_{1\le j\le n}\bigg\{ f\bigg(\frac{t(\bu)^{-\beta_j}}{|\bu|}\bigg)\bigg(\frac{t(\bu)^{-\beta_j}}{|\bu|}\bigg)^{(1-m)n}\prod_{\ell=j}^nt(\bu)^{\beta_j-\beta_\ell}\bigg\}.\]

By Corollary \ref{c:multivariable} (2), we have
\[\begin{split}
	&D_{\balpha,\bbeta}^\bb(\psi)^c\\
	\subset& \bigg\{A\in[0,1]^{mn}:\|A_{*,j}\cdot\bu\|_\Z<c t(\bu)^{-\beta_j}\text{ ($1\le j\le n$) for infinitely many $\bu\in\Z^m$}\bigg\},
\end{split}\]
where $c=((m+n)!)^2\cdot 2^{\tau|\bbeta|}$.
With this reformulation, the convergence part of the theorem follows easily by considering the natural cover of the set on the right-hand side.
To obtain the Hausdorff measure, we rewrite it as
\[\bigcap_{U=1}^\infty\bigcup_{u=U}^\infty\bigcup_{\bu\in \Z^m:|\bu|=u}\bigcup_{\bv\in\Z^n:|\bv|\le u}\prod_{j=1}^{n}\Delta\bigg(\rc_{\bu,v_j},\frac{ct(\bu)^{-\beta_j}}{|\bu|}\bigg),\]
where for $1\le j\le m$,
\[\Delta\bigg(\rc_{\bu,v_j},\frac{ct(\bu)^{-\beta_j}}{|\bu|}\bigg):=\big\{\bx\in[0,1]^m:|\bx\cdot\bu-v_j|<ct(\bu)^{-\beta_j}\big\}.\]
Note that for $1\le j\le n$, $\Delta(\rc_{\bu,v_j},ct(\bu)^{-\beta_j}/|\bu|)$ is simply the set of points in $[0,1]^m$ whose distance to the hyperplane $\bx\cdot\bu-v_j=0$ is less than $ct(\bu)^{-\beta_j}/|\bu|$. We will cover the set
\[\prod_{j=1}^{n}\Delta\bigg(\rc_{\bu,v_j},\frac{ct(\bu)^{-\beta_j}}{|\bu|}\bigg)\]
by balls of radius $t(\bu)^{-\beta_j}/|\bu|$, for $1\le j\le n$, respectively. Let $1\le j\le n$.
In the $\ell$th direction ($1\le \ell\le n$), since $\beta_1\ge \cdots\ge\beta_n$, the $\ell$th component $\Delta(\rc_{\bu,v_\ell},ct(\bu)^{-\beta_\ell}/|\bu|)$ can be covered by
\[\asymp\begin{dcases}
	\bigg(\frac{t(\bu)^{-\beta_j}}{|\bu|}\bigg)^{1-m}&\text{ if $\ell<j$}\\
	\frac{t(\bu)^{-\beta_\ell}}{|\bu|}\bigg(\frac{t(\bu)^{-\beta_j}}{|\bu|}\bigg)^{-m}&\text{ if $\ell\ge j$}
\end{dcases}\]
balls of radius
\[\frac{t(\bu)^{-\beta_j}}{|\bu|}.\]
By varying $j$, we observe that the Hausdorff $f$-content of $\prod_{j=1}^n\Delta(\rc_{\bu,v_j},ct(\bu)^{-\beta_j}/|\bu|)$ is at most
\[\begin{split}
	&\ll\min_{1\le j\le n}\bigg\{ f\bigg(\frac{t(\bu)^{-\beta_j}}{|\bu|}\bigg)\prod_{\ell=1}^{j-1}\bigg(\frac{t(\bu)^{-\beta_j}}{|\bu|}\bigg)^{1-m}\prod_{\ell=j}^n\bigg(	\frac{t(\bu)^{-\beta_\ell}}{|\bu|}\bigg(\frac{t(\bu)^{-\beta_j}}{|\bu|}\bigg)^{-m}\bigg)\bigg\}\\
	&= \min_{1\le j\le n}\bigg\{ f\bigg(\frac{t(\bu)^{-\beta_j}}{|\bu|}\bigg)\bigg(\frac{t(\bu)^{-\beta_j}}{|\bu|}\bigg)^{(1-m)n}\prod_{\ell=j}^nt(\bu)^{\beta_j-\beta_\ell}\bigg\}=\gamma_\bu(\bbeta,f).
\end{split}\]
Note that for any $\bu\in\Z^m\setminus\{0\}$, there are $\asymp |\bu|^{n}$ choices of $\bv\in\Z^n$ for which $|\bv|\le |\bu|$.
By the definition of Hausdorff $f$-measure and the convergence of the series in \eqref{eq:concon},
\[\begin{split}
	&\hm^f\big(D_{\balpha,\bbeta}^\bb(\psi)^c\big)	\le \liminf_{u\to\infty}\sum_{\bu\in\Z^m:|\bu|\ge u}\gamma_\bu(\bbeta,f)|\bu|^n=0.
\end{split}\]

\section{Divergence part of Theorem \ref{t:main}}\label{s:diver}
In this section, assume that $f$ is a dimension function such that
\[f\prec nm,\quad (mn-a)\preceq f\preceq(mn-a+1) \text{ for some $1\le a\le n-1$}\]
and
\begin{equation}\label{eq:divcon}
	\sum_{\bu\in\Z^m\setminus\{0\}}\gamma_\bu(\bbeta,f)|\bu|^n=\infty,
\end{equation}
where
\[\gamma_\bu(\bbeta,f)=\min_{1\le j\le n}\bigg\{ f\bigg(\frac{t(\bu)^{-\beta_j}}{|\bu|}\bigg)\bigg(\frac{t(\bu)^{-\beta_j}}{|\bu|}\bigg)^{(1-m)n}\prod_{\ell=j}^nt(\bu)^{\beta_j-\beta_\ell}\bigg\}.\]
Assume further that there exists a constant $\lambda>1$ such that for any $t\ge 1$,
\begin{equation}\label{eq:lambda''}
	\frac{\psi(t)}{\psi(2t)}\ge\lambda>1.
\end{equation}

For $\bb\in\R^m\setminus\Z^m$, define
\begin{equation}\label{eq:epsilonbdef}
	\epsilon(\bb):=\min_{1\le i\le m,~\|b_i\|_\Z>0}\frac{\|b_i\|_\Z}{4}.
\end{equation}
Clearly, $\epsilon(\bb)>0$ since $\bb\notin\Z^m$.
Recall that $\tau(\bb,\bu)$ (see Corollary \ref{c:multivariable} (1)) is the smallest integer for which
\[\lambda^{-\tau(\bb,\bu)\alpha_i}<\frac{\|\bu\cdot\bb\|_\Z}{m+n}\quad\text{for $1\le i\le m$.}\]
It is straightforward to verify that there exists a constant $c_\bb>0$ such that
\[\tau(\bb,\bu)\le c_\bb\]
holds for all $\bu\in\Z^m$ with $\|\bu \cdot \bb\|_\Z > \epsilon(\bb)$.
Let $\tilde{c}:=\epsilon(\bb)2^{-c_\bb|\bbeta|}/(m+n)$. By Corollary \ref{c:multivariable} (1), we have
\[\begin{split}
	D_{\balpha,\bbeta}^\bb(\psi)^c
	\supset\big\{A\in[0,1]^{mn}:&\|A_{*,j}\cdot\bu\|_\Z<\tilde c t(\bu)^{-\beta_j}\text{ ($1\le j\le n$) for infinitely} \\
	&\text{many $\bu\in\Z^m$ with $\|\bu \cdot \bb\|_\Z > \epsilon(\bb)$}\big\}=:W_{n,m}^\bb(\bbeta).
\end{split}\]

The next step is to prove that $W_{n,m}^\bb(\bbeta)$ carries full Hausdorff $f$-measure. Our strategy proceeds in two steps:
\begin{enumerate}
	\item First, we define a new $\limsup$ set $W_{n,m}(\Phi)$ (see \eqref{eq:Wmnphi} for the definition), which has full $mn$-dimensional Lebesgue measure and contains $W_{n,m}^\bb(\bbeta)$. This construction will be carried out in Section~\ref{ss:fullsubset}.
	\item Second, we reformulate $W_{n,m}(\Phi)$ as a $\limsup$ set defined by balls, and apply the mass transference principle from balls to rectangles (see Theorem \ref{t:weaken}) to deduce that the full Lebesgue measure statement for $W_{n,m}(\Phi)$ implies that a shrunk $\limsup$ set contained in $W_{n,m}^\bb(\bbeta)$ has full Hausdorff $f$-measure. This argument will be presented in Section~\ref{ss:completing}.
\end{enumerate}

\subsection{Construction of a limsup set with full Lebesgue measure}\label{ss:fullsubset}
Since $f\prec mn$, by the definition of the symbol $\prec$, we have for any $\bu\in\Z^m\setminus\{0\}$ with $|\bu|$ large enough,
\[f\bigg(\frac{t(\bu)^{-\beta_j}}{|\bu|}\bigg)>\bigg(\frac{t(\bu)^{-\beta_j}}{|\bu|}\bigg)^{mn},\quad\text{for }1\le j\le n.\]
Consequently,
\begin{equation}\label{eq:t>}
	\gamma_\bu(\bbeta,f)
	>\min_{1\le j\le n}\bigg\{\bigg(\frac{t(\bu)^{-\beta_j}}{|\bu|}\bigg)^{n}\prod_{\ell=j}^nt(\bu)^{\beta_j-\beta_\ell}\bigg\}\ge\prod_{j=1}^n\frac{t(\bu)^{-\beta_j}}{|\bu|}.
\end{equation}
Since any $\limsup$ set does not depend on finitely many $\bu\in\Z^m\setminus\{0 \}$, here and hereafter we will always assume that \eqref{eq:t>} holds for all $\bu\in\Z^m\setminus\{0\}$.

Since $\beta_1\ge\beta_2\ge\cdots\ge\beta_n$, one has
\[\prod_{j=1}^n\frac{\tilde{c}t(\bu)^{-\beta_j}}{|\bu|}\le\cdots\le\bigg(\frac{\tilde{c}t(\bu)^{-\beta_k}}{|\bu|}\bigg)^k\prod_{j=k+1}^n\frac{\tilde{c}t(\bu)^{-\beta_j}}{|\bu|}\le\cdots\le\bigg(\frac{\tilde{c}t(\bu)^{-\beta_n}}{|\bu|}\bigg)^n.\]
By \eqref{eq:t>}, there exists a unique integer $1\le k=k(\bu)\le n$ for which
\begin{equation}\label{eq:unique}
	\bigg(\frac{\tilde{c}t(\bu)^{-\beta_k}}{|\bu|}\bigg)^k\prod_{j=k+1}^n\frac{\tilde{c}t(\bu)^{-\beta_j}}{|\bu|}\le \gamma_\bu(\bbeta,f)<\bigg(\frac{\tilde{c}t(\bu)^{-\beta_{k+1}}}{|\bu|}\bigg)^{k+1}\prod_{j=k+2}^n\frac{\tilde{c}t(\bu)^{-\beta_j}}{|\bu|}.
\end{equation}
We ignore the right-hand side if $k=n$. Set
\begin{equation}\label{eq:piu}
	\varpi_{\bu}=\bigg(\gamma_\bu(\bbeta,f)\prod_{j=k+1}^n\frac{|\bu|}{\tilde{c}t(\bu)^{-\beta_j}}\bigg)^{1/k}.
\end{equation}
Then, by \eqref{eq:unique},
\begin{equation}\label{eq:<o<}
	\frac{\tilde{c}t(\bu)^{-\beta_k}}{|\bu|}\le\varpi_{\bu}=\bigg(\gamma_\bu(\bbeta,f)\prod_{j=k+1}^n\frac{|\bu|}{\tilde{c}t(\bu)^{-\beta_j}}\bigg)^{1/k}<\frac{\tilde{c}t(\bu)^{-\beta_{k+1}}}{|\bu|}.
\end{equation}
Define an $n$-tuple $\Phi=(\phi_1,\dots,\phi_n)$ of functions as follows: For any $\bu\in\Z^m\setminus\{0\}$, if  $\|\bu \cdot \bb\|_\Z \le \epsilon(\bb)$, then let
\[\phi_1(\bu)=\cdots=\phi_n(\bu)=0.\]
Otherwise,
\[\phi_j(\bu):=\begin{dcases}
	\tilde c t(\bu)^{-\beta_j}&\text{ if $j>k$},\\
	|\bu|\cdot\varpi_{\bu}&\text{ if $j\le k$},
\end{dcases}\]
where $k=k(\bu)$ is as defined in \eqref{eq:unique}.
It follows from \eqref{eq:<o<} that
\begin{equation}\label{eq:mon}
	\phi_1(\bu)=\cdots=\phi_k(\bu)=|\bu|\cdot\varpi_{\bu}<\phi_{k+1}(\bu)\le\cdots\le\phi_n(\bu)
\end{equation}
and
\begin{align}
	W_{n,m}^\bb(\bbeta)
	\subset \bigg\{A\in[0,1]^{mn}:&\|A_{*,j}\cdot\bu\|_\Z<\phi_j(\bu)\text{ ($1\le j\le n$) for infinitely}\notag\\ &\text{many $\bu\in\Z^m$}\bigg\}
	=:W_{n,m}(\Phi).\label{eq:Wmnphi}
\end{align}

The rest of this subsection is devoted to proving the following result.
\begin{lem}
	Let $\Phi$ be defined as above. Then,
	\[\lm^{mn}\big(W_{n,m}(\Phi)\big)=1.\]
\end{lem}

The proof of the above lemma is divided into several steps.
 The following result, which does not carry any monotonicity
assumptions and allows us to reduce the proof to show that $W_{n,m}(\Phi)$ is of positive measure.
\begin{thm}[{\cite[Theorem 3.1]{Li13}}]\label{l:reduce}
	We have
	\[\lm^{mn}\big(W_{n,m}(\Phi)\big)\in\{0,1\}.\]
	In particular,
	\[\lm^{mn}\big(W_{n,m}(\Phi)\big)>0\quad\Longrightarrow\quad \lm^{mn}\big(W_{n,m}(\Phi)\big)=1.\]
\end{thm}
Lamperti's result \cite{La63} provides a method for establishing lower bounds on the measure of $\limsup$ sets.
\begin{lem}[\cite{La63}]\label{l:quasiind}
	Let $(X,\mu)$ be a probability space, and let $\{E_i\}$ be measurable subsets of $X$ such that
	\[\sum_{i=1}^{\infty}\mu(E_i)=\infty\qaq\mu(E_i\cap E_j)\le C\mu(E_i)\mu(E_j)\quad\text{for some $C\ge 1$}.\]
	Then,
	\[\mu\Big(\limsup_{i\to\infty}E_i\Big)>0.\]
\end{lem}
In light of Lemma \ref{l:quasiind}, the desired result follows upon showing that:
\begin{enumerate}[(a)]
	\item the sets under consideration are quasi-independent, and
	\item the sum of their measures diverges.
\end{enumerate}
To this end, for any $\bu\in\Z^m$ and $\bd=(\delta_1,\dots,\delta_m)\in(\R_{\ge 0})^m$, let
\begin{equation}\label{eq:R'}
	\begin{split}
		R'(\bu,\bd):=&\big\{A\in[0,1]^{mn}:|A_{*,j}\cdot\bu-v_j|<\delta_j\text{ ($1\le j\le n$)}\\
		&\hspace{8em}\text{for some $\bv\in\Z^n$ with $\gcd(\bu,v_j)=1$ for all $j$}\big\},
	\end{split}
\end{equation}
where $\gcd(\bu,v_j)$ denotes the greatest common divisor of $u_1,\dots,u_m,v_j$. The Lebesgue measures of the sets $R'(\bu,\bd)$ and their correlations can be estimated as follows.

\begin{lem}\label{l:mea}
	Let $\bu_1,\bu_2\in\Z^m$ satisfying $\bu_1\ne\pm \bu_2$.
	Then,
	\[\bigg(\frac{\varphi(|\bu_1|)}{|\bu_1|}\bigg)^n\prod_{j=1}^{n}\phi_j(\bu_1)\le\lm^{mn}\big(R'\big(\bu_1,\Phi(\bu_1)\big)\big)\le \prod_{j=1}^{n}\phi_j(\bu_1)\]
	and
	\[\lm^{mn}\big(R'\big(\bu_1,\Phi(\bu_1)\big)\cap R'\big(\bu_2,\Phi(\bu_2)\big)\big)\ll\prod_{j=1}^{n}\phi_j(\bu_1)\phi_j(\bu_2),\]
	where $\varphi$ denotes the Euler quotient function.
\end{lem}
The next lemma shows that a sufficiently large portion of the sets $R'(\bu, \Phi(\bu))$ possess the expected measure.
\begin{lem}[{\cite[Theorem 1]{Erd46}}]\label{l:positivedensity}
	For every sufficiently large $a>0$, the density of integers $u\in\N$ with
	\[\frac{\varphi(u)}{u}\ge\frac{1}{a}\]
	is positive and goes to $1$ as $a\to\infty$.
\end{lem}
To apply Lemma \ref{l:quasiind}, it is necessary to prove that the condition $\sum_{\bu \in \mathbb{Z}^m \setminus \{0\}} \gamma_\bu(\bbeta, f) |\bu|^n = \infty$ implies that the sum of the Lebesgue measures of certain sets diverges (see Lemma \ref{l:diverge}).
	Simple calculations show that for any $\bu\in\Z^m\setminus\{0\}$ with  $\|\bu \cdot \bb\|_\Z> \epsilon(\bb)$,
\begin{equation}\label{eq:phij=gamma}
	\prod_{j=1}^n\phi_j(\bu)=\prod_{j=k+1}^n\tilde{c} t(\bu)^{-\beta_j}\prod_{j=1}^{k}|\bu|\cdot\varpi_{\bu}=\gamma_\bu(\bbeta,f)|\bu|^n.
\end{equation}
However, this equality does not hold if $\|\bu \cdot \bb\|_\Z\le \epsilon(\bb)$, since by definition $\phi_j(\bu)=0$ but $\gamma_\bu(\bbeta,f)|\bu|^n\ne 0$. To tackle this problem, we first take a result from \cite{Kim22} that will be used to count the number of integral vectors $\bu\in\Z^m$ for which
\begin{equation}\label{eq:ub<e}
	\|\bu\cdot\bb\|_\Z\le\epsilon(\bb).
\end{equation}
\begin{lem}[{\cite[Lemma 4.4]{Kim22}}]\label{l:>epsilon b}
	For $\bb =(b_1,\dots, b_m) \in\R^m\setminus\Z^m$, let $1\le i\le m$ be an index such that $\epsilon(\bb) =\|b_i\|_\Z/4$. Then, for any $\bu\in\Z^m$, at most one of $\bu$ and $\bu+\be_i$ satisfies \eqref{eq:ub<e}, where $\be_i$ denotes the vector with a 1 in the $i$th coordinate and 0's elsewhere.
\end{lem}
For our purposes, we restate it in the following form.
\begin{cor}\label{c:>epsilonb}
	For any $\bu\in\Z^m\setminus\{0\}$, there exists $\bq\in\Z^m\setminus\{0\}$ with $|\bq-\bu|\le 1$ such that
	\[\|\bq\cdot\bb\|_\Z>\epsilon(\bb)\qaq t(\bq)\asymp t(\bu),\]
	where the implied constant depends on $\lambda$ only.
\end{cor}
\begin{proof}
	If $\bu$ satisfies $\|\bu \cdot \bb\|_\Z > \epsilon(\bb)$, then we may simply take $\bq = \bu$, and the proof is complete.

	Now suppose instead that $\|\bu \cdot \bb\|_\Z \le \epsilon(\bb)$. By Lemma~\ref{l:>epsilon b}, the integral vector $\bq = \bu + \be_i$ satisfies $\|\bq \cdot \bb\|_\Z > \epsilon(\bb)$. By the  decay assumption (see \eqref{eq:lambda''}) on $\psi$ and the definition of $t(\bu)$, it is not difficult to verify that $t(\bq)\asymp t(\bu)$.
\end{proof}
Since $\prod_{j=1}^n\phi_j(\bu)=\gamma_\bu(\bbeta,f)|\bu|^n$ (see \eqref{eq:phij=gamma}) for any $\bu\in\Z^m\setminus\{0\}$ satisfying  $\|\bu \cdot \bb\|_\Z> \epsilon(\bb)$, it is more convenient to use $\gamma_\bu(\bbeta,f)$.
\begin{lem}\label{l:elldensity}
	Let $\ell \ge 1$. Then, there are $\asymp 2^\ell$ integers $u\in [2^\ell,2^{\ell+1}]$ for which
	\[\sum_{\bu\in\Z^m:|\bu|=u}\gamma_{\bu}(\bbeta,f)\asymp \sum_{\substack{\bu\in\Z^m:|\bu|=u,\\ \|\bu \cdot \bb\|_\Z > \epsilon(\bb)}}\gamma_{\bu}(\bbeta,f).\]
\end{lem}
\begin{proof}
	The argument will be based on the following result:
	\begin{lem}\label{l:u1=u2}
		For any $u_1,u_2\in[2^\ell,2^{\ell+1}]$ with $0<u_1<u_2$,
		\begin{equation}
			\sum_{\bu_1\in\Z^m:|\bu_1|=u_1}\gamma_{\bu_1}(\bbeta,f)\asymp \sum_{\bu_2\in\Z^m:|\bu_2|=u_2}\gamma_{\bu_2}(\bbeta,f).
		\end{equation}
	\end{lem}
	\begin{proof}
		To see this, for any $\bu_1=(u_{1,1},\dots,u_{1,m})\in \Z^m$ with $|u_{1,k}|=u_1$ for some $1\le k\le m$, set
		\[\begin{split}
			\ca(\bu_1)=\bigg\{\bu_2=(u_{2,1},\dots,u_{2,m})\in\Z^m:&|u_{2,k}|=u_2\text{ and}\\
			&\text{$|u_{2,j}|\in \bigg[\frac{|u_{1,j}|u_2}{u_1},\frac{(|u_{1,j}|+1)u_2}{u_1}\bigg)$ for $j\ne k$}\bigg\}.
		\end{split}\]
		Let  $\bu_2\in \ca(\bu_1)$. It follows that
		\[\bigg|\frac{u_{2,j}}{u_{1,j}}\bigg|\asymp \frac{u_2}{u_1}\quad\text{ for every $1\le j\le m$ with $u_{1,j}\ne 0$},\]
		which, by the decay assumption (see \eqref{eq:lambda''}) on $\psi$, implies that
		\[t(\bu_2)\asymp t(\bu_1).\] Therefore,
		\[\begin{split}
			\gamma_{\bu_1}(\bbeta,f)&=\min_{1\le j\le n}\bigg\{ f\bigg(\frac{t(\bu_1)^{-\beta_j}}{|\bu_1|}\bigg)\bigg(\frac{t(\bu_1)^{-\beta_j}}{|\bu_1|}\bigg)^{(1-m)n}\prod_{\ell=j}^nt(\bu_1)^{\beta_j-\beta_\ell}\bigg\}\\
			&\asymp \min_{1\le j\le n}\bigg\{ f\bigg(\frac{t(\bu_2)^{-\beta_j}}{|\bu_2|}\bigg)\bigg(\frac{t(\bu_2)^{-\beta_j}}{|\bu_2|}\bigg)^{(1-m)n}\prod_{\ell=j}^nt(\bu_2)^{\beta_j-\beta_\ell}\bigg\}\\
			&=\gamma_{\bu_2}(\bbeta,f).
		\end{split}\]
		Notice that $\# \ca(\bu_1)\ll u_2/u_1\ll 2$.  We have
		\[\gamma_{\bu_1}(\bbeta,f)\asymp \sum_{\bu_2\in \ca(\bu_1)}\gamma_{\bu_2}(\bbeta,f).\]
		Summing over $\bu_1\in\Z^m$ with $|\bu_1|=u_1$, we prove the desired result.
	\end{proof}

{\em Now we proceed to prove Lemma \ref{l:elldensity}.}

	For any $\bu\in\Z^m$, by Corollary \ref{c:>epsilonb}, there exists $\bq=\bq(\bu)\in\Z^m\setminus\{0\}$ with $|\bq-\bu|\le 1$ such that
	\begin{equation}\label{eq:qb>epsilon}
		\|\bq\cdot\bb\|_\Z>\epsilon(\bb)\qaq t(\bq)\asymp t(\bu).
	\end{equation}
	This implies that
	\[\gamma_\bu(\bbeta,f)\asymp\gamma_\bq(\bbeta,f),\]
	and so
	\[\sum_{u=2^\ell+1}^{2^{\ell+1}-1}\sum_{\bu\in\Z^m:|\bu|=u}\gamma_{\bu}(\bbeta,f)\asymp\sum_{u=2^\ell+1}^{2^{\ell+1}-1}\sum_{\bu\in\Z^m:|\bu|=u}\gamma_{\bq(\bu)}(\bbeta,f).\]
	Since $|\bq(\bu)-\bu|\le 1$, we have $2^\ell\le |\bq(\bu)|\le 2^{\ell+1}$ and there are only $\asymp 1$  distinct integral vectors $\bu$ corresponding  to the same $\bq\in\Z^m$ satisfying \eqref{eq:qb>epsilon}. Therefore,
	\[\sum_{u=2^\ell+1}^{2^{\ell+1}-1}\sum_{\bu\in\Z^m:|\bu|=u}\gamma_{\bq(\bu)}(\bbeta,f)\ll \sum_{q=2^\ell}^{2^{\ell+1}}\sum_{\substack{\bq\in\Z^m:|\bq|=q,\\ \|\bq \cdot \bb\|_\Z > \epsilon(\bb)}}\gamma_{\bq}(\bbeta,f).\]
	By Lemma \ref{l:u1=u2}, we have
	\[\sum_{u=2^\ell}^{2^{\ell+1}}\sum_{\bu\in\Z^m:|\bu|=u}\gamma_{\bu}(\bbeta,f)\asymp \sum_{u=2^\ell+1}^{2^{\ell+1}-1}\sum_{\bu\in\Z^m:|\bu|=u}\gamma_{\bu}(\bbeta,f)\asymp \sum_{u=2^\ell}^{2^{\ell+1}}\sum_{\substack{\bu\in\Z^m:|\bu|=u,\\ \|\bu \cdot \bb\|_\Z > \epsilon(\bb)}}\gamma_{\bu}(\bbeta,f),\]
	which immediately implies the conclusion of the lemma.
\end{proof}
To apply the quasi-independent result (see Lemma \ref{l:mea}), we exclude arbitrary one of $\bu$ and $-\bu$ from the set $\{\bu\in\Z^m:|\bu|=u\text{ and }\|\bu \cdot \bb\|_\Z > \epsilon(\bb)\}$ and denote the resulting set by $\Gamma(u)$. That is, for any $\bu\in\Gamma(u)$,
\begin{equation}\label{eq:gamma}
	|\bu|=u,\quad \|\bu \cdot \bb\|_\Z > \epsilon(\bb)\quad\text{but}\quad -\bu\not\in\Gamma(u).
\end{equation}
Since by definition $t(\bu)=t(-\bu)$, one has
\begin{equation}\label{eq:gammausum}
	\sum_{\substack{\bu\in\Z^m:|\bu|=u,\\ \|\bu \cdot \bb\|_\Z > \epsilon(\bb)}}\gamma_{\bu}(\bbeta,f)\asymp \sum_{\bu\in\Gamma(u)}\gamma_{\bu}(\bbeta,f).
\end{equation}

\begin{lem}\label{l:diverge}
	There exists a set $\Lambda\subset\N$ such that
	\[\lm^{mn}\big(R'\big(\bu,\Phi(\bu)\big)\big)\asymp \prod_{j=1}^{n}\phi_j(\bu)\quad\text{for all $\bu\in\Z^m$ with $|\bu|=u\in\Lambda$}\]
	and
	\[\sum_{\bu\in\Z^m\setminus\{0\}}\gamma_\bu(\bbeta,f)|\bu|^n=\infty\quad\Longrightarrow\quad\sum_{u\in\Lambda}\sum_{\bu\in\Gamma(u)}\prod_{j=1}^{n}\phi_j(\bu)=\infty.\]
\end{lem}
\begin{proof}
	Applying Lemma \ref{l:positivedensity} with $a$ large enough, and then Lemma \ref{l:elldensity}, we obtain a set $\Lambda\subset\N$ with positive density such that for any $u\in\Lambda\cap[2^\ell,2^{\ell+1}]$ with $\ell\ge 1$ sufficiently large,
	\begin{equation}\label{eq:densitysatis}
		\frac{\varphi(u)}{u}\ge\frac{1}{a}\qaq \sum_{\bu\in\Z^m:|\bu|=u}\gamma_{\bu}(\bbeta,f)\asymp \sum_{\substack{\bu\in\Z^m:|\bu|=u,\\ \|\bu \cdot \bb\|_\Z > \epsilon(\bb)}}\gamma_{\bu}(\bbeta,f).
	\end{equation}
	The left inequality together with Lemma \ref{l:mea} yields the first point of the lemma.

	To prove the second one, observe that
	\begin{equation}\label{eq:sum=2^}
		\infty=\sum_{\bu\in\Z^m\setminus\{0\}}\gamma_\bu(\bbeta,f)|\bu|^n\asymp\sum_{\ell=1}^{\infty}\sum_{u=2^\ell}^{2^{\ell+1}}\sum_{\bu\in\Z^m:|\bu|=u}\gamma_\bu(\bbeta,f)|\bu|^n.
	\end{equation}
	For every sufficiently large $\ell$, since $\Lambda$ has positive density, by the second equation in \eqref{eq:densitysatis} and Lemma \ref{l:u1=u2},
	\[\begin{split}
		\sum_{u=2^\ell}^{2^{\ell+1}}\sum_{\bu\in\Z^m:|\bu|=u}\gamma_\bu(\bbeta,f)|\bu|^n&\asymp\sum_{u\in\Lambda\cap[2^\ell,2^{\ell+1}]}\sum_{\substack{\bu\in\Z^m:|\bu|=u,\\ \|\bu \cdot \bb\|_\Z > \epsilon(\bb)}}\gamma_{\bu}(\bbeta,f)|\bu|^n\\
		&\asymp \sum_{u\in\Lambda\cap[2^\ell,2^{\ell+1}]}\sum_{\bu\in\Gamma(u)}\gamma_{\bu}(\bbeta,f)|\bu|^n,
	\end{split}\]
	where we use \eqref{eq:gammausum} in the last step. Finally, it follows from \eqref{eq:phij=gamma} and \eqref{eq:sum=2^} that
	\[\begin{split}
		\sum_{u\in\Lambda}\sum_{\bu\in\Gamma(u)}\prod_{j=1}^{n}\phi_j(\bu)&=\sum_{u\in\Lambda}\sum_{\bu\in\Gamma(u)}\gamma_{\bu}(\bbeta,f)|\bu|^n\\
		&=\sum_{\ell=1}^{\infty}\sum_{u\in\Lambda\cap[2^\ell,2^{\ell+1}]}\sum_{\bu\in\Gamma(u)}\gamma_{\bu}(\bbeta,f)|\bu|^n\\
		&\asymp\sum_{\ell=1}^{\infty}\sum_{u=2^\ell}^{2^{\ell+1}}\sum_{\bu\in\Z^m:|\bu|=u}\gamma_\bu(\bbeta,f)|\bu|^n=\infty,
	\end{split}\]
	which completes the proof.
\end{proof}

Now we come to the main result of this subsection.
\begin{proof}[Proof of Lemma \ref{l:mea}]
	In view of Lemma \ref{l:reduce}, it suffices to prove that the $\limsup$ set has positive measure. To this end, we will show that
	\[\lm^{mn}\bigg(\limsup_{u\in\Lambda:u\to\infty}\bigcup_{\bu\in\Gamma(u)}R'\big(\bu,\Phi(\bu)\big)\bigg)>0,\]
	where $\Lambda$ is taken from Lemma \ref{l:diverge} and $\Gamma(u)$ is defined as in \eqref{eq:gamma}.
	By Lemmas \ref{l:mea} and \ref{l:diverge}, for any $\bu_1\in\Gamma(|\bu_1|)$ and $\bu_2\in\Gamma(|\bu_2|)$ with $|\bu_1|,|\bu_2|\in\Lambda$, we have
	\[\begin{split}
		&\lm^{mn}\big(R'\big(\bu_1,\Phi(\bu_1)\big)\cap R'\big(\bu_2,\Phi(\bu_2)\big)\big)
		\ll\prod_{j=1}^m\phi_j(\bu_1)\phi_j(\bu_2)\\
		\asymp& \lm^{mn}\big(R'\big(\bu_1,\Phi(\bu_1)\big)\big) \lm^{mn}\big(R'\big(\bu_2,\Phi(\bu_2)\big)\big),
	\end{split}\]
	which together with Lemma \ref{l:quasiind} concludes the lemma.
\end{proof}

\subsection{Completing the proof of the divergence part}\label{ss:completing}
Write
\[\begin{split}
	W_{n,m}(\Phi)&=\bigcap_{U=1}^\infty\bigcup_{u=U}^\infty\bigcup_{\bu\in\Z^m:|\bu|=u}\bigcup_{\bu\in\Z^n:|\bv|\le u}\prod_{j=1}^{n}\Delta\bigg(\rc_{\bu,v_j},\frac{\phi_j(\bu)}{|\bu|}\bigg).
\end{split}\]
By Lemma \ref{l:mea}, we have
\[\lm^{mn}\big(W_{n,m}(\Phi)\big)=1.\]
For each $(\bu,\bv)\in\Z^m\times\Z^n$ with $\|\bu\cdot\bb\|>\epsilon(\bb)$, we have $\phi_j(\bu)\ne 0$, and so by \eqref{eq:mon} the product set
\[\prod_{j=1}^{n}\Delta\bigg(\rc_{\bu,v_j},\frac{\phi_j(\bu)}{|\bu|}\bigg)=\prod_{j=1}^{k}\Delta(\rc_{\bu,v_j},\varpi_{\bu})\prod_{j=k+1}^{n}\Delta\bigg(\rc_{\bu,v_j},\frac{\phi_j(\bu)}{|\bu|}\bigg)\]
with $k=k(\bu)$ being defined as in \eqref{eq:unique}, can be covered by a finite collection $\mathcal C(\bu,\bv)$ of balls with radius $\varpi_{\bu}$ and centered at this set. It holds that up to a set of zero Lebesgue measure
\[W_{n,m}(\Phi)\subset\bigcap_{U=1}^\infty\bigcup_{u=U}^\infty\bigcup_{\substack{\bu\in\Z^m:|\bu|=u,\\ \|\bu \cdot \bb\|_\Z > \epsilon(\bb)}}\bigcup_{\bv\in\Z^n:|\bv|\le u}\bigcup_{B\in\mathcal{C}(\bu,\bv)}B,\]
since the union over integer vectors $\bu$ with $|\bu \cdot \bb|_\Z \le \epsilon(\bb)$ has zero Lebesgue measure.
Therefore, the $\limsup$ set defined by the balls on the right still has full Lebesgue measure. On the other hand,
\[W_{n,m}^\bb(\bbeta)\supset\bigcap_{U=1}^\infty\bigcup_{u=U}^\infty\bigcup_{\substack{\bu\in\Z^m:|\bu|=u,\\ \|\bu \cdot \bb\|_\Z > \epsilon(\bb)}}\bigcup_{\bv\in\Z^n:|\bv|\le u}\bigcup_{B\in\mathcal{C}(\bu,\bv)}\bigg(B\cap\prod_{j=1}^{n}\Delta\bigg(\rc_{\bu,v_j},\frac{\tilde c t(\bu)^{-\beta_j}}{|\bu|}\bigg)\bigg).\]
In view of Theorem \ref{t:weaken}, to conclude that $W_{n,m}^\bb(\bbeta)$ has full Hausdorff $f$-measure, it suffices to prove that for any $\mathcal C(\bu,\bv)$ and any ball $B\in\mathcal C(\bu,\bv)$,
\begin{equation}\label{eq:cont>}
	\hc^f\bigg(B\cap\prod_{j=1}^{n}\Delta\bigg(\rc_{\bu,v_j},\frac{\tilde c t(\bu)^{-\beta_j}}{|\bu|}\bigg)\bigg)\gg\lm^{mn}(B)\asymp\varpi_{\bu}^{mn}.
\end{equation}

Since the radius of $B$ is $\varpi_{\bu}$, it is not difficult to deduce from \eqref{eq:<o<} that
\[B\cap\prod_{j=1}^{n}\Delta\bigg(\rc_{\bu,v_j},\frac{\tilde c t(\bu)^{-\beta_j}}{|\bu|}\bigg)\]
contains a hyperrectangle, denoted by $R=R(B,\bu,\bv)$, whose side lengths are
\[\frac{\tilde c t(\bu)^{-\beta_1}}{|\bu|},\underbrace{\varpi_{\bu},\dots,\varpi_{\bu}}_{m-1},\frac{\tilde c t(\bu)^{-\beta_2}}{|\bu|},\underbrace{\varpi_{\bu},\dots,\varpi_{\bu}}_{m-1},\dots,\frac{\tilde c t(\bu)^{-\beta_k}}{|\bu|},\underbrace{\varpi_{\bu},\dots,\varpi_{\bu}}_{m-1},\underbrace{\varpi_{\bu},\dots,\varpi_{\bu}}_{m(n-k)}.\]
The Hausdorff $f$-content of $R$ can be estimated as follows, which together with \eqref{eq:cont>} completes the divergence part of Theorem \ref{t:main}.
\begin{lem}\label{l:content>}
	Let $R=R(B,\bu,\bv)$ be given as above. We have
	\[\hc^f(R)\gg \varpi_{\bu}^{mn},\]
	where the implied constant is absolute. Consequently,
	\[\hm^f\big(D_{\balpha,\bbeta}^\bb(\psi)^c\big)\ge \hm^f\big(W_{n,m}^\bb(\bbeta)\big)=\hm^f([0,1]^{mn}).\]
\end{lem}
\begin{proof}
	Applying Proposition \ref{p:rec} yields
	\[\begin{split}
		\hc^f(R)&\asymp \min\bigg\{\min_{1\le j\le k}\biggl\{f\bigg(\frac{\tilde c t(\bu)^{-\beta_j}}{|\bu|}\bigg)\bigg(\frac{|\bu|\cdot \varpi_{\bu}}{\tilde c t(\bu)^{-\beta_j}}\bigg)^{mn-k}\prod_{\ell=j+1}^{k}t(\bu)^{\beta_j-\beta_\ell}\biggr\},f(\varpi_{\bu})\bigg\}\\
		&\asymp\min\bigg\{\min_{1\le j\le k}\biggl\{f\bigg(\frac{t(\bu)^{-\beta_j}}{|\bu|}\bigg)\bigg(\frac{|\bu|\cdot \varpi_{\bu}}{t(\bu)^{-\beta_j}}\bigg)^{mn-k}\prod_{\ell=j}^{k}t(\bu)^{\beta_j-\beta_\ell}\biggr\},f(\varpi_{\bu})\bigg\}.
	\end{split}\]
	If the minimum is attained at $f(\varpi_{\bu})$, then by $f\prec mn$,
	\[f(\varpi_{\bu})\gg \varpi_{\bu}^{mn},\]
	which is what we seek.

	Now, suppose that the minimum is attained for some $1\le j\le k$.
	Direct calculations show that
	\begin{align}
		&f\bigg(\frac{t(\bu)^{-\beta_j}}{|\bu|}\bigg)\bigg(\frac{|\bu|\cdot \varpi_{\bu}}{t(\bu)^{-\beta_j}}\bigg)^{mn-k}\prod_{\ell=j}^{k}t(\bu)^{\beta_j-\beta_\ell}\notag\\
		=&\varpi_{\bu}^{mn-k}\prod_{\ell=k+1}^{n}\frac{|\bu|}{t(\bu)^{-\beta_\ell}}\cdot  f\bigg(\frac{t(\bu)^{-\beta_j}}{|\bu|}\bigg)\bigg(\frac{t(\bu)^{-\beta_j}}{ |\bu|}\bigg)^{(1-m)n}\prod_{\ell=j}^{n}t(\bu)^{\beta_j-\beta_\ell}\notag\\
		\ge &\varpi_{\bu}^{mn-k}\prod_{\ell=k+1}^{n}\frac{|\bu|}{t(\bu)^{-\beta_\ell}}\cdot \gamma_\bu(\bbeta,f)\asymp\varpi_{\bu}^{mn},\notag
	\end{align}
	where we use the definition of $\varpi_{\bu}$ in the last step (see \eqref{eq:piu}). The proof is thus complete.
\end{proof}
\subsection*{Acknowledgments}
Y. He was supported by NSFC (No. 12401108).


\begin{thebibliography}{10}

	\bibitem{BV10}
	V. Beresnevich and S. Velani. Classical metric Diophantine approximation revisited:
	the Khintchine--Groshev theorem. {\em Int. Math. Res. Not. IMRN} (2010), 69--86.

%

	\bibitem{BHS23}
	P. Bos, M. Hussain and D. Simmons. The generalised Hausdorff measure of sets of Dirichlet non-improvable numbers. {\em Proc. Amer. Math. Soc.} 151 (2023), 1823--1838.

	\bibitem{Cassels57}
	J.~W.~S. Cassels. {\it An introduction to Diophantine approximation}, Cambridge Tracts in Mathematics and Mathematical Physics, No. 45, Cambridge Univ. Press, New York, 1957.

	\bibitem{DV97}
	D. Dickinson and S. Velani. Hausdorff measure and linear forms. {\em J. Reine Angew. Math.} 490 (1997), 1--36.



\bibitem{Erd46}
P. Erd\" os. Some remarks about additive and multiplicative functions. {\em Bull. Amer. Math. Soc.} 52 (1946), 527--537.


	\bibitem{He25adv}
	Y. He. A unified approach to mass transference principle and large intersection property. {\em Adv. Math.}
	471 (2025), Paper No. 110267, 51 pp.

	\bibitem{He25}
	Y. He. Hausdorff measure and Fourier dimensions of limsup sets arising in weighted and multiplicative Diophantine approximation. Preprint (2025), 	arXiv:2504.09411.

	\bibitem{HKWW18}
	M. Hussain, D. Kleinbock, N. Wadleigh and B. Wang.
	Hausdorff measure of sets of Dirichlet non-improvable numbers. {\em Mathematika} 64 (2018), 502--518.

	\bibitem{Jarnik31}
	V.~Jarn\'{i}k. {U}ber die simultanen diophantischen approximationen. {\em Math. Z.} 33 (1931), 505--543.







	\bibitem{Kh24}
	A. Khintchine. Einige S\"atze \"uber Kettenbr\"uche, mit Anwendungen auf die Theorie der Diophantischen Approximationen. {\em Math. Ann.} 92 (1924), 115--125.


	\bibitem{Kim22}
	T. Kim and W. Kim. Hausdorff measure of sets of Dirichlet non-improvable affine forms. {\em Adv. Math.} 403 (2022), Paper No. 108353, 39 pp.

	\bibitem{KSS22}
	D. Kleinbock, A. Str\"ombergsson and S. Yu. A measure estimate in geometry of numbers and improvements to Dirichlet's theorem. {\em Proc. Lond. Math. Soc. (3)} 125 (2022), 778--824.

	\bibitem{KW18}
	D. Kleinbock and N. Wadleigh. A zero-one law for improvements to Dirichlet's Theorem. {\em Proc. Amer. Math. Soc.} 146 (2018), 1833--1844.

		\bibitem{KR21}
	H. Koivusalo and M. Rams. Mass transference principle: from balls to arbitrary shapes. {\em Int. Math. Res. Not. IMRN} (2021), 6315--6330.

	\bibitem{La63}
	J. Lamperti. Wiener's test and Markov chains. {\em J. Math. Anal. Appl.} 6 (1963), 58--66.

	\bibitem{LLVWZ25}
	B. Li, L. Liao, S. Velani, B. Wang and E. Zorin. Diophantine approximation and the Mass Transference
	Principle: incorporating the unbounded setup. {\em Adv. Math.} 470 (2025) 110248, 61 pp.

	\bibitem{Li13}
	L. Li. Zero-one laws in simultaneous and multiplicative Diophantine approximation. {\em Mathematika} 59 (2013), 321--332.

	\bibitem{WW21}
	B. Wang and J. Wu. Mass transference principle from rectangles to rectangles in Diophantine approximation. {\em Math. Ann.} 381 (2021), 243--317.


	\bibitem{Zh21}
	W. Zhong. Mass transference principle: from balls to arbitrary shapes: measure theory. {\em
		J. Math. Anal. Appl.} 495 (2021), Paper No. 124691, 23 pp.






\end{thebibliography}
\end{document}